\documentclass[11pt]{article}

\usepackage[all]{xy}
\usepackage{graphicx}
\usepackage{amssymb,latexsym,amsmath,amsthm, mathrsfs}
\usepackage{fullpage}
\usepackage{color}
\newtheorem{theorem}{Theorem}[section]
\newtheorem{lemma}[theorem]{Lemma}

\newtheorem{definition}[theorem]{Definition}

\newtheorem{proposition}[theorem]{Proposition}
\newtheorem{corollary}[theorem]{Corollary}

\newtheorem{conjecture}[theorem]{Conjecture}

 \DeclareMathSymbol{\N}{\mathbin}{AMSb}{"4E}

\DeclareMathSymbol{\Z}{\mathbin}{AMSb}{"5A}
\DeclareMathSymbol{\R}{\mathbin}{AMSb}{"52}
\DeclareMathSymbol{\Q}{\mathbin}{AMSb}{"51}
\DeclareMathSymbol{\I}{\mathbin}{AMSb}{"49}
\DeclareMathSymbol{\C}{\mathbin}{AMSb}{"43}

\def\P{{\mathbb P}}

\def\M{{\mathbb M}}

\numberwithin{equation}{section}

\title{Bott-Chern homology, Bott-Chern differential cohomology and the Hodge conjecture}
\author{Jyh-Haur Teh, Chin-Jui Yang}
\date{}

\begin{document}
\maketitle

\begin{abstract}
We propose a version of the Hodge conjecture in Bott-Chern cohomology and using results from characterizing real holomorphic chains by
real rectifiable currents to provide a proof for this question. We define a Bott-Chern differential cohomology and use atomic section theory
of Harvey and Lawson to construct refined Bott-Chern classes for holomorphic vector bundles in this differential cohomology. These refined Bott-Chern
classes transform naturally to standard Chern classes, Bott-Chern classes and Cheeger-Simons' refined Chern classes.
\end{abstract}

\section{Introduction}
Let $M$ be a compact oriented smooth manifold without boundary of dimension $n$. We denote by $A^k(M)$ and $\mathscr{D}'_k(M)$ the spaces of complex-valued smooth $k$-forms
and complex-valued $k$-currents on $M$ respectively.
By abuse of notation, we also denote by $d$ the operator on currents dual to the exterior derivative $d$ on smooth differential forms on $M$. Then $(\mathscr{D}_*(M), d)$ is a chain complex and $(A^*(M), d)$ is a cochain complex.
The $k$-th homology group
$$H^{DR}_k(M; \C):=H_k(\mathscr{D}'_*(M), d)$$
is called the $k$-th de Rham homology group of $M$ and the $k$-th cohomology group
$$H^k_{DR}(M; \C):=H^k(A^*(M), d)$$
is called the $k$-th de Rham cohomology of $M$.
The natural pairing $\mathscr{D}'_k(M)\times A^k(M) \rightarrow \C$ given by
$$(T, \phi)\mapsto T(\phi)$$
induces a nondegenerate bilinear pairing between de Rham homology and cohomology
$$H_k^{DR}(M; \C)\times H^k_{DR}(M; \C) \rightarrow \C$$
The natural inclusion $A^k(M) \hookrightarrow \mathscr{D}'_{n-k}(M)$ given by
$$(\phi)(\psi):=\int_M\phi\wedge \psi$$
for $\psi\in A^{n-k}(M)$ induces an isomorphism
$$PD:H^k_{DR}(M; \C) \rightarrow H^{DR}_{n-k}(M; \C)$$
which is called the Poincar\'e duality.

Throughout this paper, we let $G$ be a subgroup of $\R$. Let $f:U \rightarrow M$ be a Lipschitz map where $U$ is an open subset of some $\R^N$.
For a simplicial $r$-chain with $G$-coefficients $\sigma=\sum^m_{i=1}g_i\sigma_i$, we mean that $\{\sigma_i\}^m_{i=1}$ is a collection of non-overlapping $r$-simplexes on $U$, i.e., if they meet, they only meet at their boundaries, then we say that $f_*\sigma$ is a Lipschitz $r$-chain with $G$-coefficients.

Let $Lip_k(M; G)\subset \mathscr{D}'_k(M)$ be the subgroup of Lipschitz $k$-chains with $G$-coefficients on $M$.
Then $(Lip_*(M; G), d)$ is a chain complex. We denote its $k$-th homology group by
$$H_k(M; G):=H_k(Lip_*(M; G), d)$$
It is well known that there is a group isomorphism
$$H_k(M; G)\cong H^{\mbox{sing}}_k(M; G)$$
where the latter is the $k$-th singular homology group of $M$ with $G$-coefficients.

For a complex manifold $X$, we may consider smooth $(p, q)$-forms $A^{p, q}(X)$ and their dual
$\mathscr{D}'_{p, q}(X)$, the $(p, q)$-currents on $X$. The $\overline{\partial}$ operator gives
a cochain complex $(A^{p, *}(X), \overline{\partial})$ and its dual chain complex $(\mathscr{D}'_{p, *}(X), \overline{\partial})$.
Let
$$H^{p, q}(X):=H^q(A^{p, *}(X), \overline{\partial})$$ be the $(p, q)$-Dolbeault cohomology of $X$ and
$$H_{p, q}(X):=H_q(\mathscr{D}'_{p, *}(X), \overline{\partial})$$ be the $(p, q)$-Dolbeault homology of $X$.
For $X$ compact, we have a refinement of the Poincar\'e duality
$$PD:H^{p, q}(X) \rightarrow H_{n-p, n-q}(X)$$

Now let $X$ be a compact K\"ahler manifold of dimension $n$. One of the most important result from Hodge theory is the
Hodge decomposition on cohomology groups:
$$H^k(X; \C)\cong \bigoplus_{p+q=k}H^{p, q}(X)$$
For notation simplicity, we identify $H^{p, q}(X)$ with its image in $H^k(X; \C)$.
For $0\leq p\leq n$, by the Poincar\'e duality, we have
$$H_{2(n-p)}(X; \Q) \hookrightarrow H^{DR}_{2(n-p)}(X; \C) \overset{PD^{-1}}{\longrightarrow}H^{2p}(X; \C)$$
Define
$$H^{p, p}(X; \Q):=PD^{-1}(H_{2(n-p)}(X; \Q))\cap H^{p, p}(X)$$
where the intersection is taking place inside the ambient space $H^{2p}(X; \C)$.

For $X$ a complex projective manifold, an algebraic $p$-cycle on $X$ is a formal linear combination $\sum^m_{i=1}a_iV_i$
where each $V_i\subset X$ is an irreducible $p$ dimensional subvariety and $a_i\in \Z$. Since it is well known that compact algebraic varieties
can be smoothly triangulated, an algebraic $p$-cycle determines an integral Lipschitz $2p$-chains. We say that an element
$\alpha\in H^{p, p}(X; \Q)$ is representable by algebraic cycles with rational coefficients if there is an algebraic
$p$-cycle $R$ with rational coefficients such that
$$PD^{-1}([R])=\alpha$$
The following conjecture is one of the center questions in algebraic geometry.

\begin{conjecture}(The Hodge conjecture)
Let $X$ be a complex projective manifold of dimension $n$. Every element of $H^{p, p}(X; \Q)$ is representable by some algebraic
$(n-p)$-cycles with $\Q$-coefficients.
\end{conjecture}

In the following, we will propose a version of the Hodge conjecture on Bott-Chern cohomology. Let $Y$ be a complex manifold
of dimension $n$. The $(p, q)$-Bott-Chern cohomology group of $Y$ which has its origin in the paper \cite{BC} of Bott and Chern is
$$H^{p, q}_{BC}(Y):=\frac{\{\phi\in A^{p, q}(Y): \partial \phi=0, \overline{\partial}\phi=0\}}{Im\partial\overline{\partial}\cap A^{p, q}(Y)}$$
where $Im\partial\overline{\partial}$ is the image of the operator $\partial\overline{\partial}:A^{p+q-2}(Y) \rightarrow A^{p+q}(Y)$.
The $k$-th Bott-Chern cohomology of $Y$ is
$$H^k_{BC}(Y; \C):=\frac{\{\phi\in A^k(Y): \partial \phi=0, \overline{\partial}\phi=0\}}{Im\partial\overline{\partial}\cap A^k(Y)}$$

By definition, we have
$$H^k_{BC}(Y; \C)\cong \bigoplus_{p+q=k}H^{p, q}_{BC}(Y)$$

Dually, we define the Bott-Chern homology of $Y$ to be
$$H_{p, q}^{BC}(Y):=\frac{\{T\in \mathscr{D}'_{p, q}(Y): \partial T=0, \overline{\partial}T=0\}}{Im\partial\overline{\partial}\cap \mathscr{D}'_{p, q}(Y)}, \ \
H_k^{BC}(Y; \C):=\frac{\{T\in \mathscr{D}'_k(Y): \partial T=0, \overline{\partial}T=0\}}{Im\partial\overline{\partial}\cap \mathscr{D}'_k(Y)}$$
where $Im\partial\overline{\partial}$ is the image of the operator $\partial\overline{\partial}:\mathscr{D}'_{k+2}(Y) \rightarrow \mathscr{D}'_k(Y)$.
The $k$-th Bott-Chern homology with $\Q$-coefficients of $Y$ is
$$H^{BC}_k(Y; \Q):=\frac{\{T\in Lip_k(Y; \Q): \partial T=0, \overline{\partial}T=0\}}{Im\partial\overline{\partial}\cap Lip_k(Y; \Q)}$$
For $Y$ a compact complex manifold, it is not difficult to see that the natural pairing
$H^{p, q}_{BC}(Y)\times H^{BC}_{p, q}(Y) \rightarrow \C$ given by
$$([\phi], [T])\mapsto T(\phi)$$
is bilinear and nondegenerate, hence we have Poincar\'e duality isomorphisms
$$PD:H^{p, q}_{BC}(Y) \rightarrow H^{BC}_{n-p, n-q}(Y) \  \mbox{ and } \ PD:H^k_{BC}(Y; \C) \rightarrow H^{BC}_{2n-k}(Y; \C)$$
defined similarly as in the case of de Rham cohomology. Note that
$$H^{BC}_k(Y; \Q) \hookrightarrow H^{BC}_k(Y; \C)$$ We define
$$H^{p, p}_{BC}(Y; \Q):=PD^{-1}(H^{BC}_{2(n-p)}(Y; \Q))\cap H^{p, p}_{BC}(Y)$$
We recall that a closed subset $A\subset Y$ is called a holomorphic subvariety of $Y$ if for any $a\in A$, there is an open neighborhood $U\subset Y$
of $a$ and holomorphic functions $f_1, ..., f_m\in \mathcal{O}(U)$ such that
$$A\cap U=\{z\in U: f_1(z)=\cdots =f_m(z)=0\}$$
A fundamental fact about holomorphic subvarieties is the following Lelong's theorem.

\begin{theorem}(Lelong's theorem)
If $R$ is a holomorphic subvariety of pure dimension $p$ on some complex manifold, then $R$ defines
a $d$-closed current of type $(p, p)$.
\end{theorem}

A formal linear combination $\sum^{\infty}_{i=1}a_iV_i$ of irreducible holomorphic subvarieties of dimension $p$ is called a holomorphic $p$-chain with
$G$-coefficients if all $a_i\in G$ and $\bigcup^{\infty}_{i=1}V_i$ is a holomorphic subvarieties of dimension $p$ of $Y$. Again, since compact
holomorphic subvarieties can be smoothly triangulated, they determine a Lipschitz chain.

Now we are able to ask a question analogous to the Hodge conjecture in Bott-Chern cohomology.

\begin{conjecture}(The Hodge conjecture in Bott-Chern cohomology)
Let $Y$ be a compact complex manifold. Every element of $H^{p, p}_{BC}(Y; \Q)$ is representable by holomorphic $p$-chains with $\Q$-coefficients.
\end{conjecture}

This conjecture is not restricted to complex projective manifolds. A very surprising result is that this conjecture is actually true.
We reminder the reader that a similar version of the Hodge conjecture in de Rham cohomology is false for general compact K\"ahler manifolds.
In the following, we will give a review of a proof of this result. Before that, we give a diagram to illustrate relations between Bott-Chern
and de Rham homology/cohomology for compact K\"ahler manifolds.

\xymatrix{
& H^{p, p}_{BC}(X; \Q) \ar[rd] \ar[ld] \ar@{^{(}->}[rrr]&                             &                           & H^{p, p}(X; \Q) \ar[rd] \ar[ld]& \\
H^{p, p}_{BC}(X) \ar[rd] \ar@{.>}@/^2pc/[rrrrr]^{\cong}&              & \ar[ld] H^{BC}_{2(n-p)}(X; \Q) \ar@{^{(}->}[r]& H_{2(n-p)}(X; \Q) \ar[rd]&            & H^{p, p}(X)\ar[ld]\\
                 &  H^{2p}_{BC}(X; \C)   \ar[rrr]^{\cong}&                             &                         &  H^{2p}(X; \C)  &
}

\section{Real rectifiable currents and characterization of holomorphic chains}
Let $M$ be an oriented Riemannian manifold and $G\subset \R$ be a subgroup.
We say that an $r$-current $T$ on $M$ is $G$-rectifiable if for any $\varepsilon>0$, there is some
open set $U\subset \R^N$ for some $N\in \N$, a Lipschitz map $f:U \rightarrow M$ and some finite
polyhedral $r$-chain $\sigma$ with $G$-coefficients on $U$ such that
$$\mathbb{M}(f_*\sigma-T)<\varepsilon$$
where $\M$ denotes the mass norm of a current.
We say that $T\in \mathscr{D}'_r(M)$ is a locally $G$-rectifiable current if for any $x\in M$,
there is a $G$-rectifiable current $S_x$ in $M$ such that $x\notin \mbox{spt}(T-S_x)$ where $\mbox{spt}$ denotes
the support of currents.

On a complex manifold $X$, we say that a current $T$ of type $(p, p)$ on $X$ is positive if for any $\alpha_1, ..., \alpha_p\in A^{1, 0}(X)$
and nonnegative smooth function $f$ with compact support,
$$T(f(i\alpha_1\wedge \overline{\alpha}_1)\wedge \cdots \wedge (i\alpha_p\wedge \overline{\alpha}_p))\geq 0$$
It is then not too difficult to see that a holomorphic subvariety
$V$ of pure dimension $p$ is integral rectifiable, of type $(p, p)$, $d$-closed and positive.
The converse of this observation is the beautiful characterization theorem of King. A positive holomorphic chain is a holomorphic chain with
positive coefficients.

\begin{theorem}(King's theorem)
A $2p$-current $T$ on a complex manifold $X$ is a positive holomorphic $p$-chain if and only if $T$ is a positive, $d$-closed and
locally integral rectifiable current of type $(p, p)$.
\end{theorem}

We need a geometrical fact about holomorphic subvarieties. The symbol $\mathcal{H}$ is for the Hausdorff measure and $\Omega(2k)$ is
the volume of the unit solid closed ball in $\R^{2k}$.

\begin{theorem}(Stolenzenberg's theorem)
If $U$ is an open subset of $\C^n$ and $A$ is a holomorphic subvariety of pure dimension $k$ of $U$, then for every $a\in U$ and
all sufficiently small $r$,
$$\mathcal{H}^{2k}(B_r(a)\cap A)\geq \Omega(2k)r^{2k}$$
\end{theorem}

From this result, it is not difficult to imagine that if holomorphic subvarieties are not getting too close, their union is again a holomorphic subvariety.

\begin{proposition}\label{union of subvarieties}
Let $U\subset \C^n$ be an open subset. If each $V_j$ is an irreducible holomorphic subvariety of dimension $k$ of $U$ and $\bigcup^{\infty}_{j=1}V_j$
is of $\mathcal{H}^{2k}$-locally finite, then $\bigcup^{\infty}_{j=1}V_j$ is a holomorphic $k$-subvariety of $U$.
\end{proposition}

For real rectifiable $k$-currents on $M$, an important fundamental property is that they have integral representation of the form
$$T(\phi)=\int_N<\varphi(x), \overrightarrow{T}(x)>\Theta(||T||, x)d\mathcal{H}^k(x)$$
where $\phi\in A^k(M)$, $\overrightarrow{T}(x)\in \Lambda_k(T_x\R^n)$ is some $k$-vector and $\Theta(||T||, x)$ is the density of the total variation measure $||T||$ at $x$ and
$$N:=\{x\in M:\Theta(||T||, x)>0\}$$

One of our main tools in developing our results is the following Siu's big theorem. We use $n(T, x)$ to denote the Lelong number of a current $T$ at $x$.

\begin{theorem}(Siu's semicontinuity theorem)
Let $T$ be a $d$-closed positive $(p, p)$ current on some complex manifold $X$. The set
$$E_c:=\{x\in X: n(T, x)\geq c\}$$
is a holomorphic subvariety of dimension $\leq p$.
\end{theorem}

We give a generalization of King's theorem (see \cite{TY1, TY2}).
\begin{theorem}(Teh-Yang)
Let $X$ be a complex manifold and $T\in \mathscr{D}'_{2k}(X)$. Then $T$ is a positive real holomorphic $k$-chain if and only if $T$ is a locally real rectifiable current of type
$(k, k)$, $d$-closed, positive and $spt(T)$ is of $\mathcal{H}^{2k}$-locally finite.
\end{theorem}

We only need to prove the converse. We give a sketch of a proof:
\begin{enumerate}
\item From the integral representation
$$T(\phi)=\int_N<\varphi(x), \overrightarrow{T}(x)>\Theta(||T||, x)d\mathcal{H}^k(x)$$
we see that $N\subset \mbox{spt}(T)$ and hence is of $\mathcal{H}^{2k}$-locally finite.

\item Since $T$ is positive and real rectifiable, the Lelong number
$$n(T, x)=\Theta(||T||, x)$$ for $x\in X$.

\item By Siu's semicontinuity theorem,
$$E_{\frac{1}{n}}=\{x\in X: \Theta(||T||, x)\geq \frac{1}{n}\}$$ is a holomorphic subvariety of dimension $\leq k$.

\item Write $$N=\bigcup^{\infty}_{n=1}E_{\frac{1}{n}}=N_1\cup N_2$$
where
$$N_1=\bigcup^{\infty}_{j=1}V_j$$ is the union of all irreducible holomorphic subvarieties of $\{E_{\frac{1}{n}}\}^{\infty}_{n=1}$ of pure dimension $k$,
and $N_2$ is the union of all irreducible holomorphic subvarieties of $\{E_{\frac{1}{n}}\}^{\infty}_{n=1}$ of dimension $<k$.
Since $\mathcal{H}^{2k}(N_2)=0$, we have
$$T(\phi)=\int_{N_1}<\varphi(x), \overrightarrow{T}(x)>\Theta(||T||, x)d\mathcal{H}^k(x)$$
and hence $\mbox{spt}(T)\subset \overline{N}_1$.

\item Since $N_1\subset \mbox{spt}(T)$ is of $\mathcal{H}^{2k}$-locally finite, by Proposition \ref{union of subvarieties}, $N_1$ is a holomorphic subvariety of dimension $k$.
In particular, $N_1$ is closed in $X$.

\item Since
$$N_1\subset \mbox{spt}(T)\subset\overline{N}_1=N_1$$
we have $\mbox{spt}(T)=N_1$ and hence
$$T=\sum^{\infty}_{j=1}a_jV_j$$
where each $a_j$ is the generic Lelong number of $T$ along $V_j$.
\end{enumerate}

Our result implies King's theorem since for a positive integral rectifiable $(k, k)$-current $T$,
$\Theta(||T||, x)\in \N\cup \{0\}$ and $\mbox{spt}(T)$ is of $\mathcal{H}^{2k}$-locally finite.
Our proof avoids several technical results in coordinates choosing in King's paper. With the help of Siu's semicontinuity
theorem, our proof is probably simpler conceptually.

A natural question then to ask is the characterization problem without the positivity assumption. After 3 years of King's paper published, Harvey and Shiffman (\cite{HS74}) almost answered the question.

\begin{theorem}(Harvey-Shiffman theorem)
Let $X$ be a complex manifold and $T\in \mathscr{D}_{2k}'(X)$. Then $T$ is a holomorphic $k$-chain if and only if
$T$ is locally integral rectifiable, $d$-closed, type $(k, k)$ and $\mathcal{H}^{2k+1}(\mbox{spt}(T))=0$.
\end{theorem}

But they were unhappy with the assumption on the support of $T$. They conjectured that assumption is unnecessary. This question turned out
to be very difficult. After more than twenty years, this problem was finally solved by Alexander (\cite{A}).

\begin{theorem}(Alexander theorem)
Let $X$ be a complex manifold and $T\in \mathscr{D}_{2k}'(X)$. Then $T$ is a holomorphic $k$-chain if and only if
$T$ is locally integral rectifiable, $d$-closed and of type $(k, k)$.
\end{theorem}

We also consider a similar question for real rectifiable currents (see \cite{TY2}).

\begin{theorem}(Teh-Yang)
Let $X$ be a complex manifold and $T\in \mathscr{D}'_{2k}(X)$. Then $T$ is a real holomorphic $k$-chain if and only if
$T$ is locally real rectifiable of type $(k, k)$, $d$-closed and $\mbox{spt}(T)$ is of $\mathcal{H}^{2k}$-locally finite.
\end{theorem}

We sketch a proof by induction on $k$:
\begin{enumerate}
\item When $k=1$, we observe that Alexander's proof can be modified for our case.

\item For $k>1$, since this is a local problem, we may assume that $T\in RR_{k, k}(\Delta)$ where
$\Delta:=\Delta(0, r)$ is the open ball of radius $r$ centered at the origin of $\C^n$.

\item Since $T$ is real rectifiable, we have an integral representation of the form
$$T(\phi)=\int_W<\phi(x), \overrightarrow{T}(x)>\theta(x)d\mathcal{H}^{2k}(x)$$
where $W\subset \Delta$ is some $\mathcal{H}^{2k}$-measurable set.
Let $\xi(z):=\pm \theta(z)\overrightarrow{T}(z)$ which is a simple $2k$-vector that represents the naturally
oriented complex tangent plane to $W$ at $z$ for $\mathcal{H}^{2k}$-a.e. $z\in W$ and we choose the $\pm$ signs
so that
$$T'(\phi):=\int_W<\phi(z), \xi(z)>d\mathcal{H}^{2k}(z)$$
is a positive $(k, k)$-current. Note that $\mbox{spt}(T)=\mbox{spt}(T')$.

\item Let $\pi_j: \Delta \rightarrow \C$ be the projection to the $j$-th coordinate. From some basic properties of Federer's theory of slices, we have
$<T, \pi_j, a>\in RR^{loc}_{k-1, k-1}(\Delta)$, $\mbox{spt}<T, \pi_j, a>\subset \mbox{spt}(T)\cap \pi^{-1}_j(a)$ and
$$d<T, \pi_j, a>=<dT, \pi_j, a>=0$$
for almost all $a\in B_r(0)\subset \C$. Also, by the inequality
$$\int^*_{B_r(0)}\mathcal{H}^{2k-2}(\mbox{spt}(T)\cap \pi^{-1}_j(a))d\mathcal{H}^2(a)\leq \frac{\Omega(2k-2)\Omega(2)}{\Omega(2k)}\mathcal{H}^{2k}(\mbox{spt}(T))<\infty$$
This implies that $\mbox{spt}(T)\cap \pi^{-1}_j(a)$ is $\mathcal{H}^{2k-2}$-locally finite for almost
all $a\in B_r(0)$. By the induction hypothesis, we know that the slice $<T, \pi_j, a>$ is a real holomorphic $(k-1)$-chain for almost all $a\in B_r(0)$.

\item Let $<T, \pi_j, a>'$ be the positive $(k-1, k-1)$-current associated to $<T, \pi_j, a>$. Since
$$<T, \pi_j, a>'=<T', \pi_j, a>$$
for almost all $a\in B_r(0)$ (see \cite{A}), $<T', \pi_j, a>$ is a positive real holomorphic $(k-1)$-chain for almost all $a\in B_r(0)$.

\item To show that $dT'=0$, we only need to check that $T'(d\phi)=0$ for $\phi\in A^{k-1, k}(\Delta)$ since it is similar for $\phi\in A^{k, k-1}(\Delta)$.
From the equality
\begin{align*}
4dz_j\wedge d\overline{z}_{\ell}&=(dz_j+dz_{\ell})\wedge \overline{(dz_j+dz_{\ell})}-(dz_j-dz_{\ell})\wedge \overline{(dz_j-dz_{\ell})}\\
&+i(dz_j+idz_{\ell})\wedge \overline{(dz_j+idz_{\ell})}-i(dz_j-idz_{\ell})\wedge \overline{(dz_j-idz_{\ell})}
\end{align*}
not difficult to see that by some change of variables, it is enough to consider $\phi$ of the form
$$\phi=\omega_{j_1}\wedge \psi$$
where $\omega_{j_1}=dz_{j_1}\wedge d\overline{z}_{j_1}, \psi=fdz_{j_2}\wedge \cdots \wedge dz_{j_{k-1}}\wedge d\overline{z}_{\ell_2}\wedge \cdots \wedge d\overline{z}_{\ell_k}$
for some smooth function $f$. Then
$$T'(d\phi)=T'(\omega_{j_1}\wedge d\psi)=(T'\wedge \pi^*_{j_1}(\omega_{j_1}))(d\psi)$$
By a Federer's theorem,
$$(T'\wedge \pi^*_{j_1}(\omega_{j_1}))(d\psi)=\int_{B_r(0)}<T', \pi_{j_1}, a>(d\psi)d\mathcal{L}^2(a)=\int_{B_r(0)}(d<T', \pi_{j_1}, a>)(\psi)d\mathcal{L}^2(a)=0$$
since $<T', \pi_{j_1}, a>$ is a real holomorphic $(k-1)$-chain and hence $d$-closed.

\item Then $T'\in RR^{loc}_{k, k}(\Delta)$ is a positive, $d$-closed and $\mbox{spt}(T')$ is of $\mathcal{H}^{2k}$-locally finite.
By our result before, $T'$ is a positive real holomorphic $k$-chain and so is $T$.
\end{enumerate}

\section{Applications}
Let $M$ be a Riemannian manifold. Let $N^{loc}_k(M)$ be the space of locally normal currents of dimension $k$ on $M$.
\begin{lemma}
If $T\in N^{loc}_k(M)$ has $\mathcal{H}^k$-locally finite support, then $T$ is locally real rectifiable.
\end{lemma}

\begin{proof}
Since this is a local problem and locally real rectifiability is invariant under
isometric isomorphism, we may assume that $T$ is a normal current on $\R^n$ where $n$ is the dimension of $M$
and the support of $T$ is $\mathcal{H}^k$-finite.
Since $T$ is normal, by \cite[pg181]{KP08}, we may write $T$ as the form
$$T(\varphi) = \int_{\mbox{spt}(T)} \ \langle \varphi(x) , \vec{T}(x) \rangle \ d\| T \|(x) $$

By the assumption, $\mbox{spt}(T)$ is $\mathcal{H}^k$-finite, and by Federer's structure theorem(\cite[Thm 13.2]{S83}),
$\mbox{spt}(T)$ can be decomposed into a disjoint union
$$\mbox{spt}(T) = R \bigcup P$$
where $R$ is countably $k$-rectifiable and $P$ is purely $k$-unrectifiable and
$\mathcal{H}^k(p(P))=0$ for almost all orthogonal projections $p:\mathbb{R}^n \rightarrow \mathbb{R}^k$.
By \cite[Thm 8.5]{FF60}, $\| T \|(P) = 0$.
By \cite[pg 185]{S83}, $||T||$ is absolutely continuous with respect to $\mathcal{H}^k$, then by \cite[Remark 4.8]{S83},
$||T||=\mathcal{H}^k\lfloor \theta$
for some nonnegative $\mathcal{H}$-measurable function $\theta$ whose support is contained in $\mbox{spt}(T)$.
We have
$$T(\varphi) = \int_{R} \ \langle \varphi(x) , \vec{T}(x) \rangle \ \theta(x) d\mathcal{H}^k(x)$$
and
$$\| T \|(B) = \int_{B} \theta(x) \ d\mathcal{H}^k(x)$$
for every Borel set $B$.
Since $R$ is $\mathcal{H}^k$-finite, it is a general fact that the density
$\Theta^k(\mathcal{H}^k , R , x) = \lim_{r\rightarrow 0^+} \frac{\mathcal{H}^k(B_r(x)\bigcap R)}{\Omega(k)r^k} = 1$ for $\mathcal{H}^k$ - a.e. $x \in R$.
By \cite[pg 63]{S83}, we have
\begin{align*}
\theta(x) &= \lim_{r\rightarrow 0^+} \frac{\| T \|(B_{r}(x))}{\mathcal{H}^k(B_r(x)\bigcap R)}= \lim_{r\rightarrow 0^+} \frac{\| T \|(B_{r}(x))}{\Omega(k)r^k} \frac{\Omega(k)r^k}{\mathcal{H}^k(B_r(x)\bigcap R)}\\
&=\lim_{r\rightarrow 0^+} \frac{\| T \|(B_{r}(x))}{\Omega(k)r^k}
 = \Theta^k(\| T \| , x)
\end{align*}
exists for $\mathcal{H}^k$-a.e. $x\in R$.
Thus,
$\Theta^{\ast k}(\| T \| , x) > 0$ for $\| T \|$-a.e. $x \in \mathbb{R}^n$.
By \cite[Thm 32.1]{S83}, $T$ is real rectifiable. This completes the proof.
\end{proof}

\begin{corollary}
Let $X$ be a compact complex manifold of dimension $n$. Let $P\in Lip_{2k}(X; \Q)$. If $dP=0$ and $d^cP=0$, then
the $(k, k)$-part $P_{k, k}$ of $P$ is a holomorphic $k$-chain with $\Q$-coefficients.
\end{corollary}

\begin{proof}
Note that $\mbox{spt}(P)$ is of $\mathcal{H}^{2k}$-finite, so is $\mbox{spt}(P_{k, k})$. By the closedness assumption, $P_{k, k}$ is $d$-closed.
Since the mass $\M(P_{k, k})$ is finite, by Lemma above $P_{k, k}$ is real rectifiable. Then by our result, $P_{k, k}$ is a real holomorphic $k$-chain.
But since $P$ is a rational Lipschitz chain, $P_{k, k}$ is a holomorphic $k$-chain with $\Q$-coefficients.
\end{proof}

\begin{theorem}
Let $X$ be a compact complex manifold of dimension $n$. The Hodge conjecture of the Bott-Chern cohomology of $X$ is true.
\end{theorem}

\begin{proof}
Note that for $[e]\in H^{n-k, n-k}_{BC}(X; \Q)$,
$$e=P+dd^c a=P_{k, k}+dd^ca_{k+1, k+1}$$ for some $P\in Lip_{2k}(X; \Q)$ and some
$a\in \mathscr{D}'_{2(k+1)}(X)$. Then $P_{k, k}$ is $d$- and $d^c$-closed. By result above, $P_{k, k}$ is a holomorphic $k$-chain with
rational coefficients. This proves the Hodge conjecture of the Bott-Chern cohomology of $X$.
\end{proof}

\section{Bott-Chern differential cohomology}
Suppose that $M$ is an oriented Riemannian manifold without boundary. Let
$C_{k+1}(M)$ be the group of singular $(k+1)$-chains on $M$ and $Z_k(M)$ be the group of singular $k$-cycles on $M$.
A $k$-th Cheeger-Simons differential character is a group homomorphism
$$f:Z_k(M) \rightarrow {\mathbb R} /\Z$$
such that there is a $(k+1)$-form $w$ with

$$f(\partial \beta)\equiv \int_{\beta}w \mod \Z,  \mbox{ for all }\beta\in C_{k+1}(M)$$

The theory of differential characters was developed by Cheeger and Simons around 1970,
which can be used to detect if a Riemannian manifold can
be conformally immersed in some Euclidean spaces.

There are two important features of the differential character functor $\widehat{H}^{\bullet}$.
\begin{enumerate}
\item
$\widehat{H}^{\bullet}(M)$ is a graded-commutative ring.
\item
Existence of refined Chern classes.
\end{enumerate}

Hopkins and Singer in their paper \cite{HS} extended Cheeger-Simons differential character groups and related it to Witten's fivebrane partition functions in M-theory. About 15 years ago, Harvey and Lawson introduced a homological apparatus to study differential characters (\cite{HLZ, HL06, HL08, H}),
and gave many different presentations of differential characters. They studied a $\overline{\partial}$-analog of differential characters for complex manifolds using spark complexes and also gave a new presentation of Deligne cohomology. We refer to Freed's paper \cite{F} for more results about generalized differential cohomology.

For a holomorphic vector bundle $E$ over a complex manifold $X$ with two Hermitian metrics $h_1, h_2$, Bott and Chern (\cite{BC}) found that
the Chern forms $c_k(E, h_i)$ for $i=1, 2$ associated to the Hermitian metrics are related by the following equation
$$c_k(E, h_1)-c_k(E, h_2)=dd^cw$$
for some $w\in \Omega^{k-1, k-1}(X)$. This means that these Chern forms define a same class in the Bott-Chern cohomology of $X$.
For $h$ a Hermitian metric on $E$, we
define
$$bc_k(E):=[c_k(E, h)]\in H^{k, k}_{BC}(X)$$
and the total Bott-Chern class to be
$$bc(E):=\sum_k bc_k(E)$$

Let $\widehat{H}^k(X; \C)$ be the $k$-th differential character group with complex coefficients. A natural question is what to fill in the question marks of the following diagrams
$$\xymatrix{  \textcolor[rgb]{1.00,0.00,1.00}{?}    \ar[r] \ar[d]    & \widehat{H}^k(X; \C)   \ar[d] &  &\mbox{ and } &
\textcolor[rgb]{1.00,0.00,0.00}{?}  \ar[r] \ar[d] &   \widehat{c}_k(E, h) \ar[d] \\
H^k_{BC}(X) \ar[r] & H^k_{DR}(X; \C)    & && bc_k(E) \ar[r] &    c_k(E)} $$

Our motivation is to construct a functor that play the role of differential cohomology for Bott-Chern cohomology
and a refined Bott-Chern classes that processes all natural properties that they must have.
Since the theory of sparks introduced by Harvey and Lawson gives an easy presentation of the Cheeger-Simons' differential characters,
we follow their construction to introduce Bott-Chern sparks and use them to construct a Bott-Chern differential cohomology.
We recall that on a smooth oriented compact manifold $M$ of dimension $n$, a degree $k$ spark is a triple
$$(a, e, r)\in \mathscr{D}'_{n-k+1}(M)\oplus A^k(M) \oplus I_{n-k}(M)$$
such that
$$da=e-r$$
where $I_{n-k}(M)$ is the space of dimension $(n-k)$ integral currents on $M$.
Two degree $k$ sparks are equivalent $(a_1, e, r_1)\sim (a_2, e, r_2)$ if
$$\left\{
    \begin{array}{ll}
      a_1-a_2=db+s\\
      r_1-r_2=-ds
    \end{array}
  \right.
$$
for some $b\in \mathscr{D}'_{n-k+2}(M), s\in I_{n-k+1}(M)$. The group
$$\{\mbox{sparks of degree k on M}\}/\sim$$
also denoted by $\widehat{H}^k(M; \C)$, is the Harvey-Lawson presentation of the $k$-th Cheeger-Simons differential character group.

\section{Bott-Chern sparks}
\begin{definition}(Bott-Chern spark)
Let $X$ be a compact complex manifold of dimension $n$.
A Bott-Chern spark of degree $k$ is a triple $$(a, e, r)\in D'_{2n-k+2}(X)\oplus A^{k}(X)\oplus Lip_{2n-k}(X)$$ such that
$$\left\{
  \begin{array}{ll}
    dd^ca=e-r \\
    de=0, \ \ \ d^ce=0
  \end{array}
\right.
$$

Two Bott-Chern sparks of degree $k$ are equivalent $(a_1, e, r_1)\sim (a_2, e, r_2)$  if there exist $b\in \mathscr{D}'_{2n-k+2}(X), s\in I_{2n-k+1}(X)$ such that
$$\left\{
  \begin{array}{ll}
   d^ca_1-d^ca_2=db+s \\
    r_1-r_2=-ds \\
  \end{array}
\right.
$$

The Bott-Chern differential cohomology of degree $k$ is the group
$$\widehat{H}^k_{BC}(X):=(\mbox{Bott-Chern sparks of degree k})/\sim$$
\end{definition}

\begin{proposition}(Natural maps)
Let
$$Z^k_{BC}(X):=\{e\in A^k(X)|(a, e, r) \mbox{ is a Bott-Chern spark of degree } k \mbox{ for some } a, r\}$$
The maps
$$\xymatrix{ \widehat{H}^k_{BC}(X) \ar[r]^{\delta^{BC}_1} \ar[d]^{\delta^{BC}_2} & Z^k_{BC}(X) &\mbox{ are given by }&
[(a, e, r)]  \ar[r]^-{\delta^{BC}_1} \ar[d]^{\delta^{BC}_2} & e\\
H^k_{BC}(X; \Z)& && [r] & }$$

and there is a $3\times 3$-grid:
$$\xymatrix{ & 0 \ar[d] & 0 \ar[d] & 0 \ar[d] &\\
0 \ar[r] &  \widehat{H}^k_{BC, 0}(X) \ar[r] \ar[d]    &  \widehat{H}^k_{BC, \infty}(X)\ar[r] \ar[d] & dd^cA^{k-2}(X) \ar[r] \ar[d]  & 0\\
0\ar[r]& H^k_{BCD}(X) \ar[r] \ar[d] & \widehat{H}^k_{BC}(X) \ar[r]^{\delta^{BC}_1} \ar[d]^{\delta^{BC}_2} & Z^{k}_{BC}(X) \ar[r] \ar[d] & 0\\
0 \ar[r]& ker i^{BC}_* \ar[r] \ar[d] & H^{BC}_{2n-k}(X; \Z) \ar[r]^{i^{BC}_*} \ar[d] & H^k_{BC, I}(X) \ar[r] \ar[d] & 0\\
& 0 & 0 & 0 &\\ }$$
where $H^k_{BCD}(X)=\mbox{ kernel of } \delta^{BC}_1$, $H^k_{BC, I}(X)$ is the image of the homomorphism $i^{BC}_*$ which is the composition of homomorphisms
$H^{BC}_{2n-k}(X; \Z)\rightarrow H^{BC}_{2n-k}(X) \overset{PD^{-1}}{\longrightarrow} H^k_{BC}(X)$, $\widehat{H}^k_{BC, \infty}(X)$ is the subgroup
of $\widehat{H}^k_{BC}(X)$ which can be represented by Bott-Chern sparks of the form $(a, e, 0)$ and
$\widehat{H}^k_{BC, 0}(X)$ is the subgroup of $\widehat{H}^k_{BC}(X)$ which can be represented by Bott-Chern sparks of the form $(a, 0, 0)$.
\end{proposition}

We note that the map  $\widehat{H}^k_{BC, \infty}(X)\rightarrow dd^cA^{k-2}(X)$ follows from a fact that if $a$ is a current such that $dd^ca=e$ is
a smooth form, then there is a smooth form
$w$ such that $dd^cw=e$ (see \cite[Thm 1.2.2(i)]{GS}).

\begin{proposition}
\begin{enumerate}
\item Suppose that $X, Y$ are compact complex manifolds of dimension $n$ and $m$ respectively.
If $f:X \rightarrow Y$ is a surjective holomorphic submersion, then $f$ induces a group homomorphism
$$f_*:\widehat{H}^k_{BC}(X) \rightarrow \widehat{H}^{k+2(m-n)}_{BC}(Y)$$
given by
$$f_*[(a, e, r)]=[(f_*a, f_*e, f_*r)]$$

\item
If $X$ is a compact K\"ahler manifold, we have the following commutative diagram:
$$
\begin{xymatrix}{
\widehat{H}^k_{BC}(X) \ar[r] \ar[d] & \widehat{H}^k(X; \C)\ar[d] \\
H^k_{BC}(X) \ar[r]                  &         H^k_{DR}(X; \C)   }
\end{xymatrix}
$$
which is given by
$$
\begin{xymatrix}{
[(a, e, r)] \ar[r] \ar[d] & [(d^ca, e, r)] \ar[d]\\
[e] \ar[r] & [e]}
\end{xymatrix}
$$

\end{enumerate}
\end{proposition}

\section{Atomic sections and refined Bott-Chern classes}
To construct refined Bott-Chern classes in Bott-Chern differential cohomology for holomorphic vector bundles, we apply the atomic section theory of Harvey and Lawson.
We first give a review of some basic properties of atomic sections.

\begin{definition}(Atomic sections)
Let $M$ be a connected smooth manifold. A function $f:M \rightarrow \mathbb{R}^n$ is atomic if
$f$ is not identically zero and for each $p$-form $\frac{dy^I}{|y|^p}$ on $\R^n$ where
$I=(i_1, ..., i_k)$, $p=|I|:=\sum^k_{j=1}i_j\leq n-1$, the pullback
$$f^*(\frac{dy^I}{|y|^{|I|}})$$
to $M$ has an $L^1_{loc}$-extension across the zero set
$$Z:=\{x\in M: f(x)=0\}$$
If $f$ is atomic, the zero divisor of $f$ is the current
$$\mbox{Div}(f):=d(f^*\Theta)$$
where
$$\Theta=\frac{1}{c_k}\sum^k_{i=1}(-1)^{i-1}\frac{y_i}{|y|^k}dy_1\wedge \cdots \wedge \widehat{dy_i}\wedge \cdots \wedge dy_k$$
and $c_k=vol(S^{k-1})$ is the volume of the unit sphere in $\R^k$.
\end{definition}

It was proved by Harvey and Semmes (see \cite{HS92}) that the zero set of an atomic function has Hausdorff measure 0.
Let $E$ be a smooth oriented real vector bundle of rank $n$ over $M$. If $e=(e_1, ..., e_n)$ is a smooth local frame for $E$,
then each smooth section $s$ of $E$ determines a smooth $\R^n$-valued function $u=(u_1, ..., u_n)$ such that on this trivialization
$$s=\sum^n_{k=1}u_ke_k$$
We say that $s$ is atomic if for each choice of local frame, the corresponding function $u$ is an atomic function. The zero divisor of an atomic section
is locally defined to be $Div(u)$.

Some result of atomic sections and their divisors are given in the following (see \cite{HS92, HL93}).
\begin{proposition}
1. A holomorphic section $s$ is atomic if and only if $s$ is not identically zero.

2. Generic sections are atomic.

3. If $s$ is an atomic section,
$$Div(s)=\sum^{\infty}_{j=1}n_j[Z_j]  $$
where $\{[Z_j]\}_{j=1}^{\infty}$ is the family of integral currents associated to the connected components of $\mbox{Reg } Z$
where $Z$ is the zero locus of $s$.

4. For two atomic sections $\mu, \nu$ of a bundle $E$,
$$Div(\mu)-Div(\nu)=dR$$
for some locally rectifiable current $R$. Hence the divisors of $\mu$ and $\nu$ define a same integral cohomology class.
\end{proposition}

\begin{definition}(Atomic bundle map)
If $\alpha:E \rightarrow F$ is a bundle map between line bundle $E$ and vector bundle $F$, we say that $\alpha$ is atomic if $\alpha$ is an atomic section of $\mbox{ Hom}(E, F)$.
\end{definition}

Now let us consider the case with $k$ sections. Let $F$ be a complex vector bundle of rank $n$ over $X$.
Giving sections $\mu_0, ..., \mu_k$ of $F$ is equivalent to giving a bundle map $\alpha:\underline{\C}^{k+1} \rightarrow F$ where
$\underline{\C}^{k+1}$ is the trivial bundle over $X$ and the bundle map is defined by
$$\alpha(x, e_j):=\mu_j(x)$$
Let $\underline{\P}^k$ be the projectivization of $\underline{\C}^{k+1}$
and $\mathbb{U}$ be the tautological line bundle over $\underline{\P}^k$.

Then pullback by $\pi$, we get a diagram
$$\xymatrix{ \mathbb{U}\subset \pi^*\underline{\C}^{k+1} \ar[rr]^-{\widetilde{\alpha}} \ar[rd]&                                &\pi^*F   \ar[ld] &\underline{\C}^{k+1} \ar[rr]^-{\alpha} \ar[rd] &   &F \ar[ld]\\
                                                                                    & \underline{\P}^k \ar[rrr]^{\pi} &                &                                               & X &}
$$

\begin{definition}
We say that a family of sections $\{\mu_0, ..., \mu_k\}$ of a vector bundle $F$ is atomic if the associated bundle map $\widetilde{\alpha}$ is atomic. In this case, we define
the divisor of this family of sections to be
$$\mathbb{D}_k(\alpha):=\pi_*Div(\widetilde{\alpha})$$
\end{definition}

The following important result by Harvey and Lawson can be found in \cite[VI. Thm 1.5]{HL93}.
\begin{theorem}
Let $\alpha$ be a family of $k+1$ cross-sections which is $k$-atomic of a complex vector bundle $F$ of rank $n$ over a complex manifold $X$. Then there is a $L^1_{loc}$-current $\sigma$
on $X$, canonically defined for each choice of hermitian metric $h$ on $F$, such that
$$c_{n-k}(F, h)-\mathbb{D}_k(\alpha)=d\sigma$$
\end{theorem}
Recall that the $(n-k)$-th refined Chern class of $F$ with $h$ is
$$\widehat{c}_{n-k}(F, h):=[(\sigma, c_{n-k}(F, h), \mathbb{D}_k(\alpha))]\in \widehat{H}^k(X; \C)$$

When $F$ is a holomorphic vector bundle there is a $dd^c$-refinement of above result (see \cite[Remark 4.18]{HL95}).

\begin{theorem}($dd^c$-refinement)
Suppose that $F$ is a holomorphic vector bundle of rank $n$ over a complex manifold $X$
and $\alpha$ is a family of $k+1$ holomorphic cross-sections of $F$ which is $k$-atomic. Then there is a $L^1_{loc}$-current $\eta$
on $X$, canonically defined for each choice of hermitian metric $h$ on $F$, such that
$$c_{n-k}(F,h)-\mathbb{D}_k(\alpha)=dd^c\eta$$
\end{theorem}

\begin{definition}(Refined Bott-Chern classes)
Given a holomorphic vector bundle $F$ of rank $n$ with hermitian metric $h$ over a compact complex manifold $X$ and a family $\alpha$ of $(k+1)$-holomorphic cross-sections which is $k$-atomic,
we define
$$\widehat{bc}_{n-k}(F, h):=[(\eta, c_{n-k}(F, h), \mathbb{D}_k(\alpha)]\in \widehat{H}^{2(n-k)}_{BC}(X)$$
which is called the $(n-k)$-th refined Bott-Chern class of $F$.
\end{definition}

We need to show that this definition is independent of the choices of $h$ and $\alpha$. We use notations as in the definition above.

\begin{proposition}
Suppose that $h, h'$ are hermitian metrics on $F$, $\alpha$ and $\alpha'$ are two families of $(k+1)$-holomorphic cross-sections which are $k$-atomic, then

$$c_{n-k}(F, h)-\mathbb{D}_k(\alpha)=dd^c\eta \mbox{ \ \ and \ \ }
c_{n-k}(F, h')-\mathbb{D}_k(\alpha')=dd^c\eta'$$
Furthermore,
$$d^c\eta-d^c\eta'=dL+R$$ where $R$ is rectifiable
and
$$\mathbb{D}_k(\alpha)-\mathbb{D}_k(\alpha')=-dR$$

\end{proposition}

\begin{proof}
This follows from \cite[Theorem 7.6]{HZ01}.
\end{proof}

\begin{theorem}
Suppose that $F$ is a holomorphic vector bundle of rank $n$ over a compact K\"ahler manifold $X$.
The natural transformations in the following diagram
$$\xymatrix{ \widehat{H}^{2k}_{BC}(X) \ar[r] \ar[d] & \widehat{H}^{2k}(X; \C) \ar[d]\\
H^{2k}_{BC}(X) \ar[r] & H^{2k}_{DR}(X; \C)}$$
transform the refined Bott-Chern class $\widehat{bc}_{k}(F, h)$ to Chern class, refined Chern class and Bott-Chern class:
$$\xymatrix{ \widehat{bc}_{k}(F, h) \ar[r] \ar[d] & \widehat{c}_{k}(F, h) \ar[d]\\
bc_{k}(F) \ar[r] & c_{k}(F)}$$
\end{theorem}

From this, we have successfully defined a Bott-Chern differential cohomology and refined Bott-Chern classes for holomorphic vector bundles
over compact complex manifolds. This Bott-Chern differential cohomology can be used to distinguish compact complex manifolds
which can not be distinguished by Bott-Chern cohomology. Such examples can be found from a small deformation of the Iwasawa manifold
as in \cite{Teh16}.

\bibliographystyle{amsplain}

\begin{itemize}
\item[] Jyh-Haur Teh, Department of Mathematics, Tsing Hua University(Hsinchu), 101 KuangFu Road, Section 2, 30013 Hsinchu, Taiwan.\\
Email address: jyhhaur@math.nthu.edu.tw

\item[] Chin-Jui Yang, Department of Mathematics, Tsing Hua University(Hsinchu), 101 KuangFu Road, Section 2, 30013 Hsinchu, Taiwan.\\
Email address: p9433256@gmail.com
\end{itemize}

\end{document}